\documentclass[12pt,reqno]{article}

\usepackage[usenames]{color}
\usepackage{amssymb}
\usepackage{graphicx}
\usepackage{amscd}

\usepackage{amsthm}
\newtheorem{theorem}{Theorem}
\newtheorem{lemma}[theorem]{Lemma}
\newtheorem{proposition}[theorem]{Proposition}
\newtheorem{corollary}[theorem]{Corollary}

\theoremstyle{definition}

\newtheorem{example}[theorem]{Example}

\usepackage[colorlinks=true,
linkcolor=webgreen, filecolor=webbrown,
citecolor=webgreen]{hyperref}

\definecolor{webgreen}{rgb}{0,.5,0}
\definecolor{webbrown}{rgb}{.6,0,0}

\usepackage{color}

\usepackage{float}

\usepackage{graphics,amsmath,amssymb}
\usepackage{amsfonts}
\usepackage{latexsym}
\usepackage{epsf}

\DeclareMathOperator{\Rev}{Rev}

\newcommand{\seqnum}[1]{\href{http://www.research.att.com/cgi-bin/access.cgi/as/~njas/sequences/eisA.cgi?Anum=#1}{\underline{#1}}}

\begin{document}

\begin{center}
\vskip 1cm{\LARGE\bf On the duals of the Fibonacci and Catalan-Fibonacci polynomials and Motzkin paths} \vskip 1cm \large
Paul Barry\\
School of Science\\
Waterford Institute of Technology\\
Ireland\\
\href{mailto:pbarry@wit.ie}{\tt pbarry@wit.ie}
\end{center}
\vskip .2 in

\begin{abstract} We use the inversion of coefficient arrays to define dual polynomials to the Fibonacci and Catalan-Fibonacci polynomials, and we explore the properties of these new polynomials sequences. Many of the arrays involved are Riordan arrays. Direct links to the counting of Motzkin paths by different statistics emerge. \end{abstract}

\section{Preliminaries}
The Fibonacci polynomials are the family of polynomials $F_n(y)$ with generating function $F(x,y)=\frac{x}{1-yx-x^2}$ \cite{Hoggatt1, Hoggatt2, Ricci, Yuan}. We immediately have that $F_n(1)=F_n$, the Fibonacci numbers \seqnum{A000045}, which explains the name of this family. We have
\begin{align*} F_0(y)&=0\\
F_1(y)&=1\\
F_2(y)&=y\\
F_3(y)&=y^2+1\\
F_4(y)&=y^3+2y\\
\ldots \end{align*}
By the \emph{dual Fibonacci polynomials} $\hat{F}_n(y)$ we shall mean the polynomials whose generating function is given by the series reversion of $F(x,y)$, where the reversion is taken with respect to $x$. To find this generating function, we solve the equation
$$\frac{u}{1-yu-u^2}=x$$ to get the solution
$$u(x)=\frac{1-yx-\sqrt{1+2yx+(y^2+4)x^2}}{2x}.$$ We find that
\begin{align*}
\hat{F}_0(y)&=0\\
\hat{F}_1(y)&=1\\
\hat{F}_2(y)&=-y\\
\hat{F}_3(y)&=y^2-1\\
\hat{F}_4(y)&=-y^3+3y\\
\ldots \end{align*}
More insight is gained by characterizing the coefficient arrays of these polynomials. It will be seen that many of the coefficient arrays we meet in this note are Riordan arrays \cite{book, SGWW} or are closely related to them. Many examples of Riordan arrays are documented in the On-Line Encyclopedia of Integer Sequences (OEIS) \cite{SL1, SL2}. Sequences in this database are referenced by their $Axxxxxx$ numbers.
\begin{lemma} The coefficient array of the Fibonacci polynomial sequence $F_1(y), F_2(y), F_3(y),\ldots$ is the Riordan array $\left(\frac{1}{1-x^2}, \frac{x}{1-x^2}\right)$.
\end{lemma}
\begin{proof} By the theory of Riordan arrays, the bivariate generating function of the Riordan array $\left(\frac{1}{1-x^2}, \frac{x}{1-x^2}\right)$ is given by
$$\frac{\frac{1}{1-x^2}}{1-y\frac{x}{1-x^2}}=\frac{1}{1-yx-x^2}.$$
\end{proof}
\begin{corollary} We have
$$F_{n+1}(y)=\sum_{k=0}^n \binom{\frac{n+k}{2}}{k}\frac{1+(-1)^n}{2} y^k.$$
\end{corollary}
\begin{proof} The $(n,k)$-th element of the Riordan array $\left(\frac{1}{1-x^2}, \frac{x}{1-x^2}\right)$ is given by
$$t_{n,k}=[x^n] \frac{1}{1-x^2}\left(\frac{x}{1-x^2}\right)^k=\binom{\frac{n+k}{2}}{k}\frac{1+(-1)^n}{2}.$$
\end{proof}
This coefficient array begins
$$\left(
\begin{array}{cccccc}
 1 & 0 & 0 & 0 & 0 & 0 \\
 0 & 1 & 0 & 0 & 0 & 0 \\
 1 & 0 & 1 & 0 & 0 & 0 \\
 0 & 2 & 0 & 1 & 0 & 0 \\
 1 & 0 & 3 & 0 & 1 & 0 \\
 0 & 3 & 0 & 4 & 0 & 1 \\
\end{array}
\right).$$
The combinatorial meaning of the $(n,k)$-th element of this array is that it counts the number of ways an $n \times 1$ board can be tiled with $2 \times 1$ dominoes and exactly $k$ $1 \times 1$ squares.
The inversion of this array, denoted by $\left(\frac{1}{1-x^2}, \frac{x}{1-x^2}\right)^!$, begins
$$\left(
\begin{array}{cccccc}
 1 & 0 & 0 & 0 & 0 & 0 \\
 0 & -1 & 0 & 0 & 0 & 0 \\
 -1 & 0 & 1 & 0 & 0 & 0 \\
 0 & 3 & 0 & -1 & 0 & 0 \\
 2 & 0 & -6 & 0 & 1 & 0 \\
 0 & -10 & 0 & 10 & 0 & -1 \\
\end{array}
\right).$$
The array $\left(\frac{1}{1-x^2}, \frac{x}{1-x^2}\right)$ is an element of the Bell subgroup of the group of Riordan arrays. We therefore have the following \cite{Rinversion}.
\begin{corollary} The coefficient array of the dual Fibonacci polynomials $\hat{F}_1(y), \hat{F}_2(y), \hat{F}_3(y),\ldots$ is given by the exponential Riordan array
$$\left[\frac{I_1(2ix)}{ix}, -x\right].$$
\end{corollary}
Here, $i=\sqrt{-1}$. The general element of this array is given by
$$\hat{t}_{n,k}=\binom{n}{k}C_{\frac{n-k}{2}}(-1)^{\frac{n+k}{2}}\frac{1+(-1)^{n-k}}{2},$$ where
$C_n=\frac{1}{n+1}\binom{2n}{n}$ is the $n$-the Catalan number.
Then $$\hat{F}_{n+1}(y)= \sum_{k=0}^n \hat{t}_{n,k} y^k.$$
The corresponding matrix $\left[\frac{I_1(2x)}{x}, x\right]$ with all nonnegative elements is \seqnum{A097610} in the OEIS. This array counts the number of Motzkin paths of length $n$ having $k$ horizontal steps.
We can generalize these results by considering the generating function $\frac{1}{1-yx-zx^2}$. Expanding this along $x$ we have the following.
\begin{align*}
[x^n]\frac{1}{1-yx-zx^2}&=[x^n](1-x(y+zx))^{-1}\\
&=[x^n] \sum_{i=0}^{\infty} x^i (y+zx)^i\\
&=[x^n] \sum_{i=0}^{\infty} x^i \sum_{j=0}^i \binom{i}{j}y^j z^{i-j} x^j\\
&= \sum_{i=0}^n \binom{i}{n-i} y^{n-i} z^{2i-n}\\
&= \sum_{i=0}^n \binom{n-i}{i} y^i z^{n-2i}.\end{align*}
We then have
$$F_{n+1}(y)=\sum_{i=0}^n \binom{i}{n-i}y^{n-i} \quad\text{and}\quad F_{n+1}(y)=\sum_{i=0}^{\lfloor \frac{n}{2} \rfloor} \binom{n-i}{i}y^i.$$
Thus we have a second and a third matrix associated with the Fibonacci polynomials.

The second matrix is the lower-triangular invertible triangle $\left(\binom{k}{n-k}\right)_{0 \le n,k \le \infty}$, which corresponds to the Riordan array $(1, x(1+x))$. This triangle begins
$$\left(
\begin{array}{cccccc}
 1 & 0 & 0 & 0 & 0 & 0 \\
 0 & 1 & 0 & 0 & 0 & 0 \\
 0 & 1 & 1 & 0 & 0 & 0 \\
 0 & 0 & 2 & 1 & 0 & 0 \\
 0 & 0 & 1 & 3 & 1 & 0 \\
 0 & 0 & 0 & 3 & 4 & 1 \\
\end{array}
\right).$$
The generating function of this matrix is given by
$$\frac{1}{1-yx(1+x)}=\frac{1}{1-yx-yx^2}.$$
To get its inversion, we thus solve the equation
$$\frac{u}{1-yu-yu^2}=x$$ to get
$$\frac{u}{x}=\frac{\sqrt{1+2yx+y(y+4)x^2}-yx-1}{2yx^2}.$$ This expands to give the matrix $(1, x(1+x))^!$ that begins
$$\left(
\begin{array}{cccccc}
 1 & 0 & 0 & 0 & 0 & 0 \\
 0 & -1 & 0 & 0 & 0 & 0 \\
 0 & -1 & 1 & 0 & 0 & 0 \\
 0 & 0 & 3 & -1 & 0 & 0 \\
 0 & 0 & 2 & -6 & 1 & 0 \\
 0 & 0 & 0 & -10 & 10 & -1 \\
\end{array}
\right).$$
The general element of this matrix is $$\tilde{t}_{n,k}=\frac{(-1)^k}{k+1}\binom{n}{k}\binom{k+1}{n-k+1}.$$ The nonnegative matrix is \seqnum{A107131}, which counts Motzkin paths of length $n$ with $k$ up steps, or $k$ horizontal steps. We let $\tilde{F}_n(y)$ be the polynomials with $$\tilde{F}_0(y)=0, \tilde{F}_1(y)=1, \tilde{F}_2=-y, \tilde{F}_3(y)=y^2-y, \tilde{F}_4(y)=-y^3+3y^2,\ldots,$$ defined by the above matrix. We have the following result.
\begin{proposition}
$$\tilde{F}_{n+1}(y)=y^n\,_2F_1\left(\frac{1}{2}-\frac{n}{2}, -\frac{n}{2}; 2 ; -\frac{4}{y}\right).$$
\end{proposition}
We can express the dual polynomials $\hat{F}_n$ in terms of the matrix $(\tilde{t}_{n,k})$ as follows.
\begin{proposition}
We have
$$\hat{F}_{n+1}(y)=\sum_{k=0}^n \tilde{t}_{n,k}y^{2k-n}.$$
\end{proposition}

The third matrix associated with the Fibonacci polynomials is the matrix $\left(\binom{n-k}{k}\right)$ (which is the one most usually associated with the Fibonacci polynomials). This is the ``stretched'' Riordan array $\left(\frac{1}{1-x}, \frac{x^2}{1-x}\right)$, which begins
$$\left(
\begin{array}{cccccc}
 1 & 0 & 0 & 0 & 0 & 0 \\
 1 & 0 & 0 & 0 & 0 & 0 \\
 1 & 1 & 0 & 0 & 0 & 0 \\
 1 & 2 & 0 & 0 & 0 & 0 \\
 1 & 3 & 1 & 0 & 0 & 0 \\
 1 & 4 & 3 & 0 & 0 & 0 \\
\end{array}
\right).$$ This matrix is \seqnum{A011973} in the OEIS.
Its generating function is given by
$$\frac{\frac{1}{1-x}}{1-y\frac{x^2}{1-x}}=\frac{1}{1-x-yx^2}.$$ To find the inversion of this matrix, we solve the equation
$$\frac{u}{1-u-yu^2}=x$$ to get
$$\frac{u}{x}=\frac{\sqrt{1+2x+(1+4y)x^2}-x-1}{2yx^2}$$ as the generating function of the inversion. This expands to give the matrix $\left(\tilde{\tilde{t}}_{n,k}\right)$ that begins
$$\left(
\begin{array}{cccccc}
 1 & 0 & 0 & 0 & 0 & 0 \\
 -1 & 0 & 0 & 0 & 0 & 0 \\
 1 & -1 & 0 & 0 & 0 & 0 \\
 -1 & 3 & 0 & 0 & 0 & 0 \\
 1 & -6 & 2 & 0 & 0 & 0 \\
 -1 & 10 & -10 & 0 & 0 & 0 \\
\end{array}
\right).$$
This matrix is the coefficient array of the polynomials $\tilde{\tilde{F}}_n(y)$ with
$$\tilde{\tilde{F}}_0(y)=0, \tilde{\tilde{F}}_1(y)=1, \tilde{\tilde{F}}_2(y)=-1,\tilde{\tilde{F}}_3(y)=1-y,
\tilde{\tilde{F}}_4(y)=3y-1,\tilde{\tilde{F}}_5(y)=2y^2-6y+1,\ldots.$$ The general term of this matrix is
$$\tilde{\tilde{t}}_{n,k}=\binom{n}{2k}C_k(-1)^{n-k}.$$  The nonnegative matrix $\left(\binom{n}{2k}C_k\right)$ is \seqnum{A055151}, which counts the number of Motzkin paths of length $n$ with $k$ up steps. We an express the dual Fibonacci polynomials $\hat{F}_n(y)$ in terms of this matrix as follows.
\begin{proposition} We have  $$\hat{F}_{n+1}(y)=\sum_{k=0}^{\lfloor \frac{n}{2} \rfloor} \tilde{\tilde{t}}_{n,k}y^{n-2k}=\sum_{k=0}^{\lfloor \frac{n}{2} \rfloor} \binom{n}{2k}C_k(-1)^{n-k}y^{n-2k}.$$
\end{proposition}

\section{Catalan-Fibonacci polynomials and their duals}
The Catalan-Fibonacci polynomials are obtained by scaling the Fibonacci polynomials by the Catalan numbers.
Thus we set $CF_n(y)=C_{n-1}F_n(y)$.
In order to explore this concept, we first look at the relevant generating functions. We have the following result in this direction.
\begin{proposition} We have
$$[x^{n+1}] \Rev\left(x(\sqrt{1-4bx^2}-ax)\right)=C_n \sum_{i=0}^{\lfloor \frac{n}{2} \rfloor}\binom{n-i}{i}a^{n-2i}b^i.$$
\end{proposition}
\begin{proof}
The proof uses Lagrange Inversion \cite{H2, LI}. We have
\begin{align*}
[x^{n+1}] \Rev\left(x(\sqrt{1-4bx^2}-ax)\right)&=\frac{1}{n+1}[x^n] \left(\sqrt{1-4bx^2}-ax\right)^{-(n+1)}\\
&=\frac{1}{n+1}[x^n] \sum_{j=0}^{\infty} \binom{-(n+1)}{j}(1-4bx^2)^{\frac{j}{2}}(-ax)^{-(n+1)-j}\\
&=\frac{1}{n+1}[x^n] \sum_{j=0}^{\infty} \binom{n+j}{j}(-1)^j \sum_{i=0}^{\frac{j}{2}} \binom{\frac{j}{2}}{i}(-4b)^i x^{2i}(-ax)^{-n-j-1}\\
&=\frac{1}{n+1}\sum_{i \ge 0} \binom{\frac{2i-2n-1}{2}}{i}(-4b)^i \binom{2i-n-1}{2i-2n-1}(-a)^{n-2i}\\
&=\frac{1}{n+1} \sum_{i \ge 0} \binom{-\left(\frac{2n-2i+1}{2}\right)}{i}(-4b)^i \binom{2i-n-1}{n}(-a)^{n-2i}\\
&=\frac{1}{n+1} \sum_{i \ge 0} \binom{\frac{2n-2i+1}{2}+i-1}{i}(4b)^i \binom{-(n-2i+1)}{n}(-a)^{n-2i}\\
&=\frac{1}{n+1}  \sum_{i \ge 0} \binom{n-\frac{1}{2}}{i}(4b)^i \binom{n-2i+1+n-1}{n}(-1)^n(-a)^{n-2i}\\
&=\frac{1}{n+1} \sum_{i \ge 0} \binom{n-\frac{1}{2}}{i}\binom{2n-2i}{n}4^i a^{n-2i}b^i\\
&=\frac{1}{n+1} \sum_{i=0}^{\lfloor \frac{n}{2} \rfloor} \binom{2n}{n} \binom{n-i}{i}a^{n-2i}b^i\\
&=C_n \sum_{i=0}^{\lfloor \frac{n}{2} \rfloor} \binom{n-i}{i}a^{n-2i}b^i.\end{align*}
\end{proof}
\begin{corollary} The generating function of the Catalan-Fibonacci polynomial sequence $C_n F_{n+1}(y)$ is given by
$$\frac{1}{x}\Rev\left(x(\sqrt{1-4yx^2}-x)\right).$$
\end{corollary}
In order to get a closed expression for $\Rev\left(x(\sqrt{1-4bx^2}-ax)\right)$, we solve the equation
$$u(\sqrt{1-4bu^2}-au)=x$$ and we take the solution with $u(0)=0$. We find that
$$\Rev\left(x(\sqrt{1-4bx^2}-ax)\right)=\frac{\sqrt{1-2ax-\sqrt{1-4a x-16 bx^2}}}{\sqrt{2} \sqrt{a^2+4 b}}.$$
The following result is immediate.
\begin{corollary}
The generating function of the Catalan-Fibonacci polynomials $CF_n(y)$ is given by
$$\frac{\sqrt{1- 2x -\sqrt{1-4x -16 y x^2}}}{\sqrt{2} \sqrt{1+4 y}}.$$
\end{corollary}
Regarded as the bivariate generating function in $x$ and $y$, this generating function expands to give the matrix that begins
$$\left(
\begin{array}{cccccc}
 1 & 0 & 0 & 0 & 0 & 0 \\
 1 & 0 & 0 & 0 & 0 & 0 \\
 2 & 2 & 0 & 0 & 0 & 0 \\
 5 & 10 & 0 & 0 & 0 & 0 \\
 14 & 42 & 14 & 0 & 0 & 0 \\
 42 & 168 & 126 & 0 & 0 & 0 \\
\end{array}
\right).$$
We define the \emph{dual Catalan-Fibonacci polynomials} $\hat{FC}_n(y)$ to be the sequence of polynomials whose generating function is given by the series reversion of that of the Catalan-Fibonacci polynomials. Thus we have that the generating function of the  dual Catalan-Fibonacci polynomials is given by
$$x(\sqrt{1-4yx^2}-x).$$
These polynomials therefore start
$$0, 1, -1, - 2y, 0, - 2y^2, 0, - 4y^3, 0, - 10y^4, 0,\ldots.$$
It is interesting to note the simple form of these polynomials, which are defined essentially by the Catalan numbers, since we have
$$(2,2,4,10,\ldots)=2(1,1,2,5,\ldots).$$
In terms of the inversion of coefficient matrices, we have the following.
$$\left(
\begin{array}{cccccc}
 1 & 0 & 0 & 0 & 0 & 0 \\
 1 & 0 & 0 & 0 & 0 & 0 \\
 2 & 2 & 0 & 0 & 0 & 0 \\
 5 & 10 & 0 & 0 & 0 & 0 \\
 14 & 42 & 14 & 0 & 0 & 0 \\
 42 & 168 & 126 & 0 & 0 & 0 \\
\end{array}
\right)^!=\left(
\begin{array}{cccccc}
 1 & 0 & 0 & 0 & 0 & 0 \\
 -1 & 0 & 0 & 0 & 0 & 0 \\
 0 & -2 & 0 & 0 & 0 & 0 \\
 0 & 0 & 0 & 0 & 0 & 0 \\
 0 & 0 & -2 & 0 & 0 & 0 \\
 0 & 0 & 0 & 0 & 0 & 0 \\
\end{array}
\right).$$
\begin{example}
The sequence $\hat{CF}_{n+1}(1)$ begins
$$1, -1, -2, 0, -2, 0, -4, 0, -10, 0, -28, 0, -84, 0, -264, 0, -858,0,\ldots.$$
The Hankel transform of this sequence begins
$$1, -3, 14, -32, 96, -208, 544, -1152, 2816, -5888,\ldots.$$
This has generating function
$$\frac{1-x+4x^2}{(1-2x)(1+2x)^2}.$$
The sequence $\hat{CF}_{n+1}(-1)$ begins
$$1, -1, 2, 0, -2, 0, 4, 0, -10, 0, 28, 0, -84, 0, 264, 0, -858,0,\ldots.$$
The Hankel transform of this sequence begins
$$1, 1, -10, -16, 64, 112, -352, -640, 1792, 3328,\ldots.$$
and it has generating function
$$\frac{1+x-2x^2-8x^3}{(1+4x^2)^2}.$$
\end{example}
\section{The Catalan-Fibonacci matrix}
We have $CF_{n+1}(y)=C_n \sum_{i=0}^{\lfloor \frac{n}{2} \rfloor}\binom{n-i}{i}a^{n-2i}b^i$. The sequence
$CF_{n+1}(y)$ begins
$$1, a, 2(a^2 + b), 5a(a^2 + 2b), 14(a^4 + 3a^2b + b^2), 42a(a^4 + 4a^2b + 3b^2),\ldots.$$
In matrix terms, we can express this in two ways. We have
$$\left(
\begin{array}{cccccc}
 1 & 0 & 0 & 0 & 0 & 0 \\
 a & 0 & 0 & 0 & 0 & 0 \\
 2 a^2 & 2 & 0 & 0 & 0 & 0 \\
 5 a^3 & 10 a & 0 & 0 & 0 & 0 \\
 14 a^4 & 42 a^2 & 14 & 0 & 0 & 0 \\
 42 a^5 & 168 a^3 & 126 a & 0 & 0 & 0 \\
\end{array}
\right)
\left(
\begin{array}{c}
1\\
b\\
b^2\\
b^3\\
b^4\\
b^5 \\ \end{array}
\right)=\left(
\begin{array}{c}
1\\
a\\
2(a^2 + b)\\
5a(a^2 + 2b\\
14(a^4 + 3a^2b + b^2\\
42a(a^4 + 4a^2b + 3b^2) \\ \end{array}
\right),$$ and
$$\left(
\begin{array}{cccccc}
 1 & 0 & 0 & 0 & 0 & 0 \\
 0 & 1 & 0 & 0 & 0 & 0 \\
 2 b & 0 & 2 & 0 & 0 & 0 \\
 0 & 10 b & 0 & 5 & 0 & 0 \\
 14 b^2 & 0 & 42 b & 0 & 14 & 0 \\
 0 & 126 b^2 & 0 & 168 b & 0 & 42 \\
\end{array}
\right)
\left(
\begin{array}{c}
1\\
a\\
a^2\\
a^3\\
a^4\\
a^5 \\ \end{array}
\right)=\left(
\begin{array}{c}
1\\
a\\
2(a^2 + b)\\
5a(a^2 + 2b\\
14(a^4 + 3a^2b + b^2\\
42a(a^4 + 4a^2b + 3b^2) \\ \end{array}
\right).$$
We call the matrix for $b=1$ that begins
$$\left (
  \begin {array} {cccccc}
                     1 & 0 & 0 & 0 & 0 & 0 \\
                     0 & 1 & 0 & 0 & 0 & 0 \\
                     2 & 0 & 2 & 0 & 0 & 0 \\
                     0 & 10  & 0 & 5 & 0 & 0 \\
                     14  & 0 & 42  & 0 & 14 & 0 \\
              0 & 126  & 0 & 168  & 0 & 42 \\
    \end {array}
   \right)$$
the \emph{Catalan-Fibonacci matrix}. Its generating function is
$$\frac{\sqrt{1- 2yx - \sqrt{1-4yx-16x^2}}}{\sqrt{2(y+4)}}.$$
Its row sums are the numbers $C_nF_{n+1}$, which gives the sequence \seqnum{A098614} in the OEIS.
The inversion of the Catalan-Fibonacci matrix is the matrix that begins
$$\left(
\begin{array}{cccccc}
 1 & 0 & 0 & 0 & 0 & 0 \\
 0 & -1 & 0 & 0 & 0 & 0 \\
 -2 & 0 & 0 & 0 & 0 & 0 \\
 0 & 0 & 0 & 0 & 0 & 0 \\
 -2 & 0 & 0 & 0 & 0 & 0 \\
 0 & 0 & 0 & 0 & 0 & 0 \\
\end{array}
\right).$$
Here, the first column is the sequence
$$1, 0, -2, 0, -2, 0, -4, 0, -10, 0, -28,0,\ldots.$$

When $b=2$, we get the matrix that begins
$$\left(
\begin{array}{cccccc}
 1 & 0 & 0 & 0 & 0 & 0 \\
 0 & 1 & 0 & 0 & 0 & 0 \\
 4 & 0 & 2 & 0 & 0 & 0 \\
 0 & 20 & 0 & 5 & 0 & 0 \\
 56 & 0 & 84 & 0 & 14 & 0 \\
 0 & 504 & 0 & 336 & 0 & 42 \\
\end{array}
\right).$$ We call this the \emph{Catalan-Jacobsthal} matrix. Its row sums are the product of the Catalan numbers and the Jacobsthal numbers. This row sum sequence is sequence \seqnum{A200375} in the OEIS.

\section{The generating function $\frac{1}{\sqrt{1-4bx^2}-ax}$}
To explore the reciprocal of the generating function $\sqrt{1-4bx^2}-ax$ we consider the Riordan array
$$\left(\frac{1}{\sqrt{1-4bx^2}}, \frac{x}{\sqrt{1-4bx^2}}\right).$$
By the fundamental theorem of Riordan arrays, we have

\begin{align*}
\left(\frac{1}{\sqrt{1-4bx^2}}, \frac{x}{\sqrt{1-4bx^2}}\right)\cdot \frac{1}{1-ax}&=
\frac{1}{\sqrt{1-4bx^2}} \frac{1}{1-a \frac{x}{\sqrt{1-4bx^2}}}\\ &=
\frac{1}{\sqrt{1-4bx^2}-ax}.\end{align*}

Equivalently, we have
$$\frac{1}{\sqrt{1-4bx^2}-ax}=\left(\frac{1}{\sqrt{1-4bx^2}}, \frac{x}{\sqrt{1-4bx^2}}\right)\cdot \frac{1}{1-ax}=\left(\frac{1}{\sqrt{1-4bx^2}}, \frac{ax}{\sqrt{1-4bx^2}}\right)\cdot \frac{1}{1-x}.$$
This gives us the following result.
\begin{proposition} The generating function $\frac{1}{\sqrt{1-4bx^2}-ax}$ is the generating function of the row sums of the Riordan array $\left(\frac{1}{\sqrt{1-4bx^2}}, \frac{ax}{\sqrt{1-4bx^2}}\right)$.
\end{proposition}
The array $\left(\frac{1}{\sqrt{1-4bx^2}}, \frac{ax}{\sqrt{1-4bx^2}}\right)$ is thus the coefficient array of the bivariate polynomials in $a$ and $b$ that begin
$$1, a, a^2 + 2b, a^3 + 4ab, a^4 + 6a^2b + 6b^2, a^5 + 8a^3b + 16ab^2, a^6 + 10a^4b + 30a^2b^2 + 20b^3,\ldots.$$
Specializing to the case $b=y$ and $a=1$, which is the case of the Catalan-Fibonacci polynomials, we find that these ``reciprocal'' polynomials begin
$$1, 1, 2y + 1, 4y + 1, 6y^2 + 6y + 1, 16y^2 + 8y + 1, 20y^3 + 30y^2 + 10y + 1,\ldots.$$
The Riordan array $\left(\frac{1}{\sqrt{1-4x^2}}, \frac{x}{\sqrt{1-4x^2}}\right)$ begins
$$\left(
\begin{array}{cccccc}
 1 & 0 & 0 & 0 & 0 & 0 \\
 0 & 1 & 0 & 0 & 0 & 0 \\
 2 & 0 & 1 & 0 & 0 & 0 \\
 0 & 4 & 0 & 1 & 0 & 0 \\
 6 & 0 & 6 & 0 & 1 & 0 \\
 0 & 16 & 0 & 8 & 0 & 1 \\
\end{array}
\right).$$
This is \seqnum{A111959} in the OEIS. Since this is a Bell matrix, and since we have $\Rev\left(\frac{x}{\sqrt{1-4x^2}}\right)=\Rev\left(\frac{x}{\sqrt{1+4x^2}}\right)$, we deduce that its inversion is the exponential Riordan array
$$\left[I_0(2ix),-x\right],$$ which begins
$$\left(
\begin{array}{cccccc}
 1 & 0 & 0 & 0 & 0 & 0 \\
 0 & -1 & 0 & 0 & 0 & 0 \\
 -2 & 0 & 1 & 0 & 0 & 0 \\
 0 & 6 & 0 & -1 & 0 & 0 \\
 6 & 0 & -12 & 0 & 1 & 0 \\
 0 & -30 & 0 & 20 & 0 & -1 \\
\end{array}
\right).$$
The corresponding nonnegative matrix $\left[I_0(2x),x\right]$ is \seqnum{A109187} in the OEIS. Its elements count grand Motzkin paths of length $n$ with $k$ level steps.

\section{Conclusion} The Fibonacci polynomials are related to the number of ways we can tile an $n \times 1$ rectangle by $2 \times 1$ dominoes and $1 \times 1$ squares \cite{Benjamin}. In this paper we have indicated that the dual Fibonacci and Catalan-Fibonacci polynomials have interpretations in terms of Motzkin paths. In this optic, for instance, a Motzkin path of length $n$ with $k$ horizontal steps is ``dual'' to a tiling of the $n \times 1$ board by dominoes and exactly $k$ $1 \times 1$ square. We have used the theory of triangle inversions, and particularly Riordan array inversions, as the principal tool in this investigation.

\bigskip
\hrule

\noindent 2010 {\it Mathematics Subject Classification}: Primary
11B39; Secondary 11B83, 15B36, 11C20.
\noindent \emph{Keywords:} Fibonacci polynomials, Catalan numbers, Catalan-Fibonacci polynomials, Motzkin path, Riordan array, matrix inversion.

\bigskip
\hrule
\bigskip
\noindent (Concerned with sequences
\seqnum{A000045},
\seqnum{A011973},
\seqnum{A011973},
\seqnum{A097610},
\seqnum{A098614},
\seqnum{A107131},
\seqnum{A109187},
\seqnum{A111959}, and
\seqnum{A200375}.)

\end{document}